\documentclass[parskip=half]{scrartcl}
\usepackage[utf8]{inputenc}
\usepackage[T1]{fontenc}
\usepackage{lmodern}
\usepackage{amsfonts}
\usepackage{titletoc}
\usepackage[colorlinks=true,linkcolor=black,urlcolor=black,citecolor=black,bookmarks]{hyperref}
\usepackage{amsmath}
\usepackage{amssymb}
\usepackage{mathrsfs}
\usepackage{mathtools}
\usepackage{amsthm}
\usepackage{tabularx}
\usepackage{enumerate}
\usepackage{thmtools}
\usepackage{tikz}
\usepackage{float}
\usepackage{comment}
\usepackage{tikz-cd}
\usepackage[nameinlink,nosort]{cleveref}
\usepackage[onehalfspacing]{setspace}
\usepackage[style=alphabetic, giveninits=true, sorting=nyt,maxnames=4,isbn=false,url=false,doi=false]{biblatex}

\declaretheoremstyle[spaceabove=12pt plus 2pt minus 4pt ,spacebelow=12pt plus 2pt minus 4pt,bodyfont=\normalfont]{scdef}
\declaretheoremstyle[spaceabove=12pt plus 2pt minus 4pt,spacebelow=8pt plus 2pt minus 4pt,bodyfont=\itshape]{scthm} 
\declaretheoremstyle[spaceabove=\topsep,spacebelow=12pt plus 2pt minus 4pt,bodyfont=\normalfont,headfont=\itshape,notefont=\itshape,postheadspace=0.3em,qed=\qedsymbol,headformat=\NAME\NOTE,notebraces=\relax]{scpf}
\declaretheorem[numbered=no,style=scthm,name=Theorem, refname={Theorem,Theorems}, Refname={Theorem,Theorems}] {theorem*}
\declaretheorem[style=scthm,numberwithin=section,name=Theorem, refname={Theorem,Theorems}, Refname={Theorem,Theorems}] {theorem} 
\declaretheorem[style=scthm,sharenumber=theorem, name=Lemma, refname={Lemma,Lemmas}, Refname={Lemma,Lemma}] {lemma} 
\declaretheorem[style=scthm,sharenumber=theorem, name=Corollary, refname={Corollary,Corollaries}, Refname={Corollary,Corollaries}] {corollary}
\declaretheorem[style=scdef,sharenumber=theorem, name=Remark, refname={Remark,Remarks}, Refname={Remark,Remarks}] {remark}
\declaretheorem[style=scdef, sharenumber=theorem, name=Definition, refname={Definition,Definitions}, Refname={Definition,Definitions}]{definition}

\declaretheorem[style=scdef, sharenumber=theorem, name=Proposition, refname={Proposition,Propositions}, Refname={Proposition,Propositions}]{proposition}

\declaretheorem[style=scdef, sharenumber=theorem, name=Corollary \& Definition, refname={Corollary \& Definition,Definitions}, Refname={Corollary \& Definition,Definitions}]{cordef}

\newcommand{\Aut}{\mathrm{Aut}}

\newcommand{\mf}{\mathfrak}
\newcommand{\mc}{\mathcal}
\newcommand{\R}{\mathbb{R}}

\newcommand{\id}{\mathrm{id}}

\newcommand\extalg{%
  \newlength{\len}%
  \settoheight{\len}{V}%
  \mathbin{%
    \resizebox{0.93\len}{0.93\len}{$\wedge$}%
    \kern-0.1em%
  }}%
\newcommand{\intprod}{\mathbin{\hbox to 0.7ex{%
      \kern-0.3ex
      \vrule height0.0777ex width0.971ex depth0ex
      \kern-0.055ex
      \vrule height1.165ex width0.0777ex depth0ex\hss}}%
}%

  \DeclareMathOperator{\tr}{tr}
  \DeclareMathOperator{\rk}{rk}
  \DeclareMathOperator{\im}{im}
  
  \DeclareMathOperator{\Hess}{Hess}
  \DeclareMathOperator{\Ric}{Ric}
  
  \DeclareMathOperator{\Div}{div}
  \DeclareMathOperator{\End}{End}
  \DeclareMathOperator{\codim}{codim}

\dottedcontents{section}[2.9em]{}{2.9em}{1pc}

\title{A Bochner Technique For Foliations With Non-Negative Transverse Ricci Curvature}
\author{Leon Roschig}
\date{}

\AtEveryBibitem{\clearlist{language}}

\begin{document}

\maketitle

\begin{abstract}
\begin{center} \textbf{Abstract}\end{center}
We generalize the Bochner technique to foliations with non-negative transverse Ricci curvature. In particular, we obtain a new vanishing theorem for basic cohomology. Subsequently, we provide two natural applications, namely to degenerate 3-$(\alpha,\delta) $-Sasaki and certain Sasaki-$ \eta $-Einstein manifolds, which arise for example as Boothby-Wang bundles over hyperkähler and Calabi-Yau manifolds, respectively.
\end{abstract}

\tableofcontents

\section{Introduction}

The Bochner technique is a highly acclaimed method of proof in classical differential geometry, which is attributed to \textsc{Bochner} \cite{Boch} and sometimes also \textsc{Yano} \cite{YB}. It discusses how Ricci curvature affects the types of vector fields that a manifold admits. A modern introduction to the topic can be found in \cite[Chapter 7]{Pete}, which also served as inspiration for the generalization in this article.

Let $ (M,g) $ be a connected closed oriented Riemannian manifold and let $ \Ric $ denote the Ricci curvature tensor of $ (M,g) $. The Bochner technique culminates in the following two theorems \cite[Theorems 36 \& 48]{Pete}:

\begin{theorem} \label{negBoch}
If $ \Ric $ is negative semi-definite everywhere, then every Killing vector field is parallel. If additionally $ \Ric $ is negative definite at one point, then no non-trivial Killing vector fields exist.
\end{theorem}

\begin{theorem} \label{posBoch}
If $ \Ric $ is positive semi-definite everywhere, then every harmonic vector field (i.e.~the $ g $-dual one-form is harmonic) is parallel. If additionally $ \Ric $ is positive definite at one point, then no non-trivial harmonic vector fields exist.
\end{theorem}

By combining these results with the famous Hodge theorem, which states that the space of harmonic one-forms is isomorphic to the first cohomology group, we obtain the following interesting consequences:

\begin{corollary} \label{van}
If $ \Ric $ is positive semi-definite everywhere and positive definite at one point, then the first Betti number $ b_1(M) = 0 $.
\end{corollary}

\begin{corollary} \label{flat}
If $ \Ric $ vanishes everywhere, then the dimension of the isometry group of $ (M,g) $ is equal to $ b_1(M) $.
\end{corollary}

Over the years \Cref{negBoch,posBoch} have been adapted to work with various additional structures on $ M $, in particular \emph{Riemannian foliations}. We will go through all of the necessary preliminaries about Riemannian foliations thoroughly in \Cref{Riem}. \textsc{Kamber} and \textsc{Tondeur} devised the following analogue of \autoref{negBoch} \cite[Theorem B]{KT2}:

\begin{theorem} \label{tnegBoch} Let $ M $ be a connected closed oriented manifold, endowed with a Riemannian foliation $ (\mc{F},g) $ and let $ \Ric^T $ denote the transverse Ricci curvature of $ (\mc{F},g) $. If $ \Ric^T $ is negative semi-definite everywhere, then every transverse Killing vector field is transverse parallel. If additionally $ \Ric^T $ is negative definite at one point, then no non-trivial transverse Killing vector fields exist.
\end{theorem}

In fact, \autoref{tnegBoch} is even a generalization of \autoref{negBoch}, where the latter corresponds to the special case that $ \mc{F} $ is the trivial foliation of $ M $ by singletons. In the direction of \autoref{van}, there is the following vanishing theorem for basic cohomology, which was discovered independently by \textsc{Min-Oo}, \textsc{Ruh} and \textsc{Tondeur} \cite[Theorem 8.16]{Tond} as well as \textsc{Hebda} \cite[Theorem 1]{Hebd}:

\begin{theorem} \label{bvan}
If $ (M,\mc{F},g) $ are as in \autoref{tnegBoch} and $ \Ric^T $ is positive definite everywhere, then the first basic Betti number $ b_1(\mc{F}) = 0 $.
\end{theorem}

The main goal of this article is to close the remaining gap by developing a generalization of \autoref{posBoch}:

\begin{theorem} \label{introBoch}
Let $ M $ be a connected closed oriented manifold, endowed with a transversely oriented harmonic Riemannian foliation $ (\mc{F}, g) $. If $ \Ric^T $ is positive semi-definite everywhere, then every basic harmonic vector field is transverse parallel. If additionally $ \Ric^T $ is positive definite at one point, then no non-trivial basic harmonic vector fields exist.
\end{theorem}

By combining \Cref{tnegBoch,introBoch} with a basic Hodge theorem, we obtain the following variations of \autoref{flat} and \autoref{bvan}:

\begin{corollary}
If $(M, \mc{F},g) $ are as in \autoref{introBoch} and $ \Ric^T $ is positive semi-definite everywhere, then $ b_1(\mc{F}) \leq \codim \mc{F} $. If additionally $ \Ric^T $ is positive definite at one point, then $ b_1(\mc{F}) = 0 $.
\end{corollary}

\begin{corollary} \label{tflat}
If $(M, \mc{F},g) $ are as in \autoref{introBoch} and $ \Ric^T $ vanishes everywhere, then the dimension of the vector space $ \mf{iso}(\mc{F}) $ of transverse Killing vector fields of $ (\mc{F},g) $ is equal to $ b_1(\mc{F}) $.
\end{corollary}

Finally, we apply \autoref{tflat} to two classes of spaces which naturally naturally satisfy all of the required conditions, namely \emph{degenerate 3-$(\alpha,\delta) $-Sasaki} and certain \emph{Sasaki-$ \eta $-Einstein manifolds}:

\begin{theorem} \label{3adS}
\vspace{-.2cm} Let $ M $ be a connected closed degenerate 3-$(\alpha, \delta) $-Sasaki manifold with characteristic foliation $ \mc{F} $. Then the dimension of the automorphism group $ \Aut(M) $ is at most $ b_1(\mc{F}) + 3 $. \\
In particular, if $ M $ arises via \cite[Theorem 3.1]{GRS} as a $ T^3 $-bundle over a compact hyperkähler manifold $ N $ with integral Kähler classes, then $ \dim \Aut(M) \leq b_1(N) + 3 $.
\end{theorem}

\begin{theorem} \label{SeE}
Let $ M $ be a connected closed Sasaki-$ \eta $-Einstein manifold with transverse Calabi-Yau structure and characteristic foliation $ \mc{F} $. Then the dimension of the automorphism group $ \Aut(M) $ is at most $ b_1(\mc{F}) + 1 $. \\
In particular, if $ M $ arises as the Boothby-Wang bundle over a compact Calabi-Yau mani- fold $ N $ with integral Kähler class, then $ \dim \Aut(M) \leq b_1(N) + 1 $.
\end{theorem}

About the structure of this article: We start with a self-contained explanation of the required fundamentals about Riemannian foliations (\Cref{Riem}) and basic Hodge theory (\Cref{sHodge}). In \Cref{sBochner} we complete the proof of the main \autoref{introBoch} as well as its consequences and in \Cref{Sas} we provide the promised applications. \\

\textbf{Acknowledgements:} The author was partially supported by the German Academic Scholarship Foundation. The author thanks Oliver Goertsches for various fruitful discussion about the subject and Leander Stecker for the suggestion to apply the technique to Sasaki-$ \eta $-Einstein manifolds.

\section{Riemannian Foliations} \label{Riem}

Let $ M $ be a smooth manifold and $ (\mc{F},g_T) $ a \emph{Riemannian foliation} on $ M $. This means that $ \mc{F} $ is a foliation on $ M $ defined by an integrable subbundle $ E \subset TM $ and $ g_T $ is a \emph{transverse metric}, i.e.~a symmetric positive semi-definite $ (0,2) $-tensor field such that $ \ker g_T = E $ and $ \mc{L}_X g_T = 0 $ for all $ X \in \Gamma_\ell(E) $, where $ \Gamma_\ell $ denotes the set of all local sections of a fiber bundle. In order to avoid having to deal with quotient bundles, we shall choose and fix a so-called \emph{bundle-like metric} $ g $ on $ M $, i.e.~a Riemannian metric such that $ g(X^\perp, Y^\perp) = g_T(X,Y) $ for all $ X,Y \in TM $, where $ Z^\perp $ denotes the $ g $-orthogonal projection of $ Z \in TM $ to $ E^\perp $.

\begin{definition} \label{vfs}
The Lie algebra of \emph{foliated vector fields} and the vector space of \emph{transverse vector fields} are given by $ \mf{fol}(\mc{F}) := N_{\Gamma(TM)}\big(\Gamma(E)\big) $, the normalizer of $ \Gamma(E) $ inside $ \Gamma(TM) $, as well as $ \mf{trans}(\mc{F}) := \mf{fol}(\mc{F}) \cap \Gamma(E^\perp) $, respectively. We call a function $ f:M \to \R $ \emph{basic} if $ X(f) = 0 $ for all $ X \in \Gamma_\ell(E) $.
\end{definition}

\begin{lemma} \label{grad}
\begin{enumerate}
\item[a)] If $ X \in \mf{fol}(\mc{F}) $, then $ f := \frac{1}{2} \, g_T(X,X) $ is basic.
\item[b)] If $ f $ is basic, then its gradient (with respect to $ g $) satisfies $ \nabla f \in \mf{trans}(\mc{F}) $.
\end{enumerate}
\end{lemma}
\begin{proof}
\begin{enumerate}
\item[a)] For all $ Y \in \Gamma_\ell(E) $: $ Y(f) = g_T([Y,X],X) = 0 $.
\item[b)] First, $ 0 = X(f) = g(\nabla f, X) $ for all $ X \in \Gamma_\ell(E) $, so $ \nabla f \in \Gamma(E^\perp) $. Furthermore, for all $ X \in \Gamma(E) $, $ Y \in \Gamma_\ell(E^\perp) $:
\[ g_T([\nabla f,X],Y) = g_T(\nabla f,[X,Y]) - X\big(g_T(\nabla f,Y)\big) = [X,Y](f) - X\big(Y(f)\big) = 0 \, . \]
\end{enumerate}
\end{proof}

\begin{definition} \label{bott}
Let $ \nabla $ denote the Levi-Civita connection of $ g $. The \emph{transverse Levi-Civita} or \emph{Bott connection} $ \nabla^T $ is the connection in the vector bundle $ E^\perp $ given by
\[ \nabla^T_X Y := \begin{cases} \left(\nabla_X Y\right)^\perp & , \; X \in \Gamma_\ell(E^\perp) \, , \\ [X,Y]^\perp &, \; X \in \Gamma_\ell(E) \, . \end{cases} \]
The condition $ \mc{L}_X g_T = 0 $ ensures that $ [X,Y]^\perp $ is indeed tensorial in $ X \in \Gamma_\ell(E) $. Note that if $ Y \in \mf{trans}(\mc{F}) $, then $ \nabla^T_X Y = [X,Y]^\perp = 0 $ for all $ X \in \Gamma_\ell(E) $. The connection $ \nabla^T $ is the unique metric and torsion-free connection in $ E^\perp $ \cite[Theorem 5.9]{Tond}, i.e. for all $ X \in \Gamma_\ell(TM) $ and $ Y,Z \in \Gamma_\ell(E^\perp) $:
\begin{align*}
X\big(g_T(Y,Z)\big) &= g_T(\nabla^T_X Y,Z) + g_T(Y, \nabla^T_X Z) \, , \\ \nabla^T_Y Z - \nabla^T_Z Y &= [Y,Z]^\perp \, . 
\end{align*}
Furthermore, $ \nabla^T $ may be characterized via a Koszul formula \cite[Proposition 1.7]{KT1}, i.e.~for all $ X, Z \in \Gamma_\ell(TM) $, $ Y \in \Gamma(E^\perp) $:
\begin{align*} \label{koszul}
2g_T(\nabla^T_X Y, Z) &= X\big(g_T(Y,Z)\big) + Y\big(g_T(Z,X)\big) - Z\big(g_T(X,Y)\big) \\
&\quad + g_T([X,Y],Z) + g_T([Z,X],Y) - g_T([Y,Z],X) \, .
\end{align*}
\end{definition}

\begin{definition}
Let $ f $ be a basic function. The \emph{transverse Hessian} $ \Hess_T f $ is the symmetric $ (0,2) $-tensor field given by
\[ \Hess_T f (X,Y) := g_T(\nabla^T_X \nabla f,Y) \, , \quad X,Y \in \Gamma_\ell(TM) \, . \]
Clearly, $ \iota_X \Hess_T f = 0 $ for all $ X \in \Gamma_\ell(E) $. The \emph{transverse Laplacian} $ \Delta_T f $ is defined as
\[ \Delta_T f := \tr_g \Hess_T f = \sum_i \Hess_T f(E_i,E_i) \, , \]
where $ E_i $ is a local $ g $-orthonormal frame. The \emph{transverse Riemann curvature tensor} $ R^T $ is given by
\[ R^T(X,Y)Z := \nabla^T_X \nabla^T_Y Z - \nabla^T_Y \nabla^T_X Z - \nabla^T_{[X,Y]} Z \, , \quad X,Y \in \Gamma_\ell(TM) , \, Z \in \Gamma_\ell(E^\perp) \, . \]
Again, $ \iota_X R^T = 0 $ for all $ X \in \Gamma_\ell(E) $ \cite[Proposition 3.6]{Tond}. As usual:
\[ R^T(X,Y,Z,V) := g_T\big(R^T(X,Y)Z,V\big) \, , \quad V \in TM \, . \]
Finally, the \emph{transverse Ricci curvature} $ \Ric^T $ is defined as
\[ \Ric^T (X,X) := \tr \big(Y \mapsto R^T(Y,X)X \big) = \sum_i R^T(E_i,X,X,E_i) \, , \quad X \in \Gamma_\ell(E^\perp) \, . \]
\end{definition}

\begin{remark} \label{curv}
It is well-known that a Riemannian foliation can be characterized equi- valently via local Riemannian submersions $ \phi: U \to N $ onto a Riemannian model space $ (N,g_N) $. The transverse Riemann curvature tensor $ R^T $ then reflects the Riemann curvature tensor $ R^N $ of the local model $ N $, as made precise by the following equation \cite[Equation 5.40]{Tond}:
\[ \phi_\ast R^T(X,Y)Z = R^N(\phi_\ast X, \phi_\ast Y)\phi_\ast Z \, , \quad X,Y,Z \in \Gamma_\ell (E^\perp) \, . \]
Likewise, $ \Ric^T $ mirrors the Ricci curvature tensor $ \Ric^N $ of $ N $, viz.
\[ \Ric^T(X,X) = \Ric^N(\phi_\ast X,\phi_\ast X) \circ \phi \, , \quad X \in \Gamma_\ell(E^\perp) \, . \]
Hence, if the Riemannian foliation $ (\mc{F},g) $ was simply given by one (global) Riemannian submersion $ \phi: M \to N $ onto a Riemannian manifold $ N $, then we could just apply the ordinary Bochner technique to $ N $ instead of the more complicated approach presented here. However, the advantage of a Bochner technique for foliations is that it also works if the model space (globally) is not as well-behaved as a smooth manifold, which is often a non-trivial condition. 
\end{remark}

From now on let $ n:= \dim M $, $ p:= \rk E $ and $ q := n-p $.

\begin{definition}
We call $ \mc{F} $ \emph{transversely orientable} if $ E^\perp $ is orientable. Suppose that $ M $ is orientable and $ \mc{F} $ is tranversely orientable. Then we shall orient $ M $ and $ E^\perp $ using their \emph{Riemannian volume forms} $ \mu $ and $ \mu_T $, respectively. This means we choose a local oriented orthonormal frame $ E_1, \ldots, E_n $ of $ TM $ such that $ E_{p+1}, \ldots, E_n $ is an oriented frame of $ E^\perp $ and require $ \mu(E_1,\ldots,E_n) = \mu_T(E_{p+1},\ldots,E_n) = 1 $. \\
If $ X \in \mf{fol}(\mc{F}) $ is a foliate vector field, then $ \iota_Y \mc{L}_X \mu_T = 0 $ for all $ Y \in \Gamma_\ell(E) $. Thus, $ \mc{L}_X \mu_T $ may be viewed as a section of the vector bundle $ \Lambda^q (E^\perp)^* $, which has rank one. Hence, we can define the \emph{transverse divergence} $ \Div_T X $ as the unique function which satisfies
\[ \mc{L}_X \mu_T = \Div_T X  \cdot \mu_T \, . \]
\end{definition}

\begin{lemma} \label{dlem}
For any transverse vector field $ X \in \mf{trans}(\mc{F}) $:
\[ \Div_T X = \tr \nabla^T X \, . \]
In particular, for any basic function $ f $:
\[ \Div_T \nabla f = \Delta_T f \, . \]
\end{lemma}
\begin{proof}
Since $ X $ is transverse, we have $ \nabla^T_Y X = 0 $ for all $ Y \in \Gamma_\ell(E) $. If $ E_{p+1}, \ldots, E_n $ is an oriented local orthonormal frame of $ E^\perp $, then:
\begin{align*}
(\mc{L}_X\mu_T)(E_{p+1},\ldots,E_n) &= X\big(\mu_T(E_{p+1},\ldots,E_n)\big) - \sum_i \mu_T(E_{p+1},\ldots,[X,E_i],\ldots,E_n) \\
&= - \sum_i g_T([X,E_i],E_i) \, \mu_T(E_{p+1},\ldots,E_n) =  - \sum_i g_T([X,E_i],E_i) \, .
\end{align*}
On the other hand:
\[ \tr \nabla^T X = \sum_i g_T(\nabla_{E_i}^T X, E_i) = \sum_i g_T(\nabla_X^T E_i, E_i) - g_T([X,E_i],E_i) =  - \sum_i g_T([X,E_i],E_i) \, . \]
\end{proof}

\begin{definition}
A foliation $ \mc{F} $ is called \emph{taut} if there exists a Riemannian metric $ g $ on $ M $ such that the leaves of $ \mc{F} $ are minimal submanifolds of $ M $ with respect to $ g $. If a Riemannian foliation is taut, then $ g $ may be chosen to be bundle-like \cite[Proposition~7.6]{Tond}, \linebreak in which case $ (\mc{F},g) $ is called \emph{harmonic}.
\end{definition}

One key tool for us will be following transverse divergence \mbox{theorem \cite[Theorem 4.35]{Tond}:}

\begin{theorem} \label{div}
Let $ M $ be a closed oriented manifold, endowed with a transversely orien- ted harmonic Riemannian foliation $ (\mc{F}, g) $. Then for any foliate vector field \mbox{$ X \in \mf{fol}(M) $:}
\[ \int_M \Div_T X \cdot \mu = 0 \, . \]
\end{theorem}

\section{Basic Hodge Theory} \label{sHodge}

Let $ M $ be a smooth manifold, endowed with a foliation $ \mc{F} $ of codimension $ q $ defined by an integrable subbundle $ E \subset TM $.

\begin{definition}
A differential $ k $-form $ \omega \in \Omega^k(M) $ is called \emph{basic} if $ \iota_X \omega = 0 $ as well as $ \mc{L}_X \omega = \iota_X d\omega = 0 $ for all $ X \in \Gamma_\ell(E) $. Note that for $ f \in \Omega^0(M) $ this coincides with the \autoref{vfs} of a basic function. \\
If $ \omega $ is basic, then so is $ d\omega $, meaning the basic differential forms constitute a subcomplex $ \Omega_B(\mc{F}) $ of the de Rham complex $ \Omega(M) $. Clearly, $ \Omega_B^k(\mc{F}) = 0 $ for $ k > q $. We denote the restriction of $ d $ to $ \Omega_B(\mc{F}) $ by $ d_B $. The cohomology ring of the complex $ (\Omega_B(\mc{F}), d_B) $ is called the \emph{basic cohomology of $ \mc{F} $} and will be denoted by $ H_B(\mc{F}) $. The \emph{basic Betti numbers of $ \mc{F} $} are defined as $ b_k(\mc{F}) := \dim H_B^k(\mc{F}) $.
\end{definition}

The inclusion $ \Omega_B^1(\mc{F}) \to \Omega^1(M) $ induces an injective map $ H_B^1(\mc{F}) \to H^1(M) $ \cite[Proposition 4.1]{Tond}. Furthermore, if $ M $ is closed and $ (\mc{F},g_T) $ is a Riemannian foliation on $ M $, then $ b_k(\mc{F}) < \infty $ for $ k=0, \ldots, q $ \cite[Chapter 4]{Tond}. \\

From now on we assume that $ M $ is closed and oriented, $ (\mc{F},g_T) $ is a transversely oriented Riemannian foliation on $ M $ and $ g $ is a bundle-like metric compatible with $ g_T $. As usual, the metric $ g $ induces an inner product on $ \Lambda^k T_x^*M $ for every $ x \in M $. We let $ \mu \in \Omega^n(M) $ denote the Riemannian volume form of $ (M,g) $ and consider the inner product $ \langle \cdot, \cdot \rangle $ on $ \Omega^k(M) $ given by
\[ \langle \omega,\omega'\rangle := \int_M g(\omega, \omega') \cdot \mu \, , \quad \omega, \omega' \in \Omega^k(M) \, . \]
We write $ \langle \cdot, \cdot\rangle_B $ for the restriction of $ \langle \cdot, \cdot\rangle $ to the subspace $ \Omega_B^k(\mc{F}) \subset \Omega^k(M) $.

\begin{definition}
The \emph{basic codifferential} $ \delta_B: \Omega^k_B(\mc{F}) \to \Omega_B^{k-1}(\mc{F}) $ is the formal adjoint of $ d_B: \Omega_B^{k-1}(\mc{F}) \to \Omega_B^k(\mc{F}) $ with respect to $ \langle \cdot, \cdot\rangle_B $, viz.
\[ \langle d_B\omega, \eta\rangle_B = \langle \omega, \delta_B \eta\rangle_B \, , \quad \omega \in \Omega_B^{k-1}(\mc{F}), \, \eta \in \Omega^k_B(\mc{F}) \, . \]
The \emph{basic Laplacian} is given by
\[ \Delta_B := d_B \delta_B + \delta_B d_B: \Omega_B^k(\mc{F}) \to \Omega_B^k(\mc{F}) \, . \]
A basic form $ \omega \in \Omega_B^k(\mc{F}) $ is called \emph{basic harmonic} if $ \Delta_B \omega = 0 $. The vector space of all basic harmonic $ k $-forms will be denoted by $ \mc{H}_B^k(\mc{F}) $.
\end{definition}

Beware that $ \Delta_B $ is not the restriction of the ordinary Laplacian $ \Delta = d\delta+\delta d $ to $ \Omega_B^k(\mc{F}) $ \cite[Equation 7.28]{Tond}. On basic functions $ \Delta_B $ differs from the previously defined transverse Laplacian $ \Delta_T $ by a sign, see \autoref{laps}.

By definition of $ \delta_B $, every $ \omega \in \Omega_B^k(\mc{F}) $ satisfies
\[ \langle \Delta_B\omega,\omega\rangle_B = \langle d_B\delta_B\omega,\omega\rangle_B + \langle \delta_B d_B\omega,\omega\rangle_B = \langle d_B\omega,d_B\omega \rangle_B + \langle \delta_B \omega, \delta_B \omega \rangle_B \, . \]
Therefore $ \Delta_B \omega = 0 $ if and only if both $ d_B \omega = 0 $ and $ \delta_B \omega = 0 $. In particular, we have a natural map $ \mc{H}_B^k(\mc{F}) \to H_B^k(\mc{F}) $. In case the bundle-like metric is chosen appropriately, there is the following basic Hodge theorem \cite[Theorem 7.51]{Tond}:

\begin{theorem} \label{Hodge}
Let $ M $ be a closed oriented manifold, endowed with a transversely orien- ted harmonic Riemannian foliation $ (\mc{F}, g) $. Then the natural map $ \mc{H}_B^k(\mc{F}) \to H_B^k(\mc{F}) $ is an isomorphism.
\end{theorem}

In preparation for the Bochner technique in the next section, we now specialize to one-forms: Recall the usual one-to-one correspondence between vector fields $ X \in \Gamma(TM) $ and their $ g $-dual one-forms $ \omega_X := \iota_X g \in \Omega^1(M) $. One easily checks that $ X \in \mf{trans}(\mc{F}) $ if and only if $ \omega_X \in \Omega_B^1(\mc{F}) $.

\begin{lemma} \label{symm}
We have $ d_B\omega_X = 0 $ if and only if $ \nabla^T X $ is $ g_T $-symmetric, i.e.
\[ g_T(\nabla^T_Y X,Z) = g_T(Y,\nabla^T_Z X) \, , \quad Y, Z \in \Gamma_\ell(TM) \, . \]
\end{lemma}
\begin{proof}
The Koszul formula from \autoref{bott} can be rewritten as
\[ 2g_T(\nabla^T_Y X,Z) = (d_B\omega_X)(Y,Z) + (\mc{L}_X g_T)(Y,Z) \, , \quad Y,Z \in \Gamma_\ell(TM) \, . \]
Since $ d_B\omega_X $ is skew-symmetric and $ \mc{L}_X g_T $ is symmetric, this yields the claim.
\end{proof}

\begin{lemma} \label{codif}
If $ (\mc{F},g) $ is harmonic, then $ \delta_B \omega_X = - \Div_T X $.
\end{lemma}
\begin{proof}
By definition of $ \delta_B $, we need to show that for all basic functions $ f \in \Omega_B^0(\mc{F}) $:
\[ \int_M g(d_B f, \omega_X) \cdot \mu = - \int_M f \cdot \Div_T X \cdot \mu \, . \]
Using \autoref{dlem} we calculate:
\[ \Div_T (f \cdot X) = f \cdot \Div_T X + g(\nabla f,X) = f \cdot \Div_T X + g(d_Bf,\omega_X) \, . \]
If we integrate over $ M $, then the left-hand side vanishes by \autoref{div}, since $ f \cdot X $ is foliate and $ (\mc{F},g) $ is harmonic, and the claim follows.
\end{proof}

\begin{remark} \label{laps}
\cref{dlem,codif} imply that for all basic functions $ f $:
\[ \Delta_B f = \delta_B d_B f = \delta_B \omega_{\nabla f} = - \Div_T \nabla f = - \Delta_T f \, .  \]
\end{remark}

\begin{cordef} \label{cordef}
If $ (\mc{F},g) $ is harmonic, then $ \omega_X $ is basic harmonic if and only if $ \nabla^T X $ is $ g_T $-symmetric and $ \Div_T X = 0 $. In this case we call $ X $ \emph{basic harmonic}.
\end{cordef}

\section{A Bochner Technique} \label{sBochner}

From now on, let $ M $ be a connected closed oriented manifold, endowed with a transversely oriented harmonic Riemannian foliation $ (\mc{F}, g) $ of codimension $ q $.

\begin{definition}
A transverse vector field $ X \in \mf{trans}(\mc{F}) $ is \emph{transverse parallel} if \mbox{$ \nabla^T X = 0 $.}
\end{definition}

Beware that a transverse vector field $ X \in \mf{trans}(\mc{F}) $ which is parallel in the usual sense that $ \nabla X = 0 $ is also transverse parallel, but the converse is not true. By virtue of \autoref{dlem} and \autoref{cordef}, every transverse parallel vector field is basic harmonic.

\begin{lemma} \label{const}
Transverse parallel vector fields have constant length. Hence, they are uniquely determined by their value at one point.
\end{lemma}
\begin{proof}
If $ X \in \mf{trans}(\mc{F}) $ is transverse parallel and $ f := \frac{1}{2} g(X,X) = \frac{1}{2} g_T(X,X) $, then for all $ Y \in \Gamma_\ell(TM) $:
\[ Y(f) = g_T(\nabla^T_Y X,X) = 0 . \]
Since $ M $ is connected, this implies that $ f $ is constant.
\end{proof}

For an endomorphism field $ A \in \Gamma(\End(TM)) $, we set
\[ |A|^2 := \tr (A \circ A^*) = \sum_i g\big(A(E_i), A(E_i)\big) \, , \]
where $ E_1, \ldots, E_n $ is a local orthonormal frame.

\begin{proposition} \label{harmeq}
Let $ X \in \mf{trans}(\mc{F}) $ be a basic harmonic vector field and consider the basic function $ f := \frac{1}{2} \, g_T(X,X) $. Then:
\begin{enumerate}
\item[a)] $ \nabla f = \nabla_X^T X $.
\item[b)] $ \Hess_T f (Y,Y) = g_T(\nabla_Y^T X, \nabla_Y^T X) + R^T(Y,X,X,Y) + g_T(\nabla^T_X \nabla^T_Y X,Y) - g_T(\nabla^T_{\nabla^T_X Y} X, Y) $ for all $ Y \in \Gamma_\ell(E^\perp) $.
\item[c)] $ \Delta_T f = |\nabla^T X|^2 + \Ric^T (X,X) $.
\end{enumerate}
\end{proposition}
\begin{proof}
\begin{enumerate}
\item[a)] By virtue of \autoref{symm}:
\[ g(\nabla f, Y) = Y(f) = g_T(\nabla^T_Y X, X) =  g_T(\nabla^T_X X, Y) = g(\nabla^T_X X, Y) \, , \quad Y \in \Gamma_\ell(TM) \, . \]
\item[b)] Part a) and \autoref{symm} imply that for all $ Y \in \Gamma_\ell(E^\perp) $:
\begin{align*}
\Hess_T f(Y,Y) &= g_T(\nabla^T_Y \nabla f, Y) = g_T(\nabla_Y^T \nabla^T_X X, Y) \\
&= R^T(Y,X,X,Y) + g_T(\nabla^T_X \nabla_Y^T X,Y) + g_T(\nabla^T_{[Y,X]} X,Y) \\
&= R^T(Y,X,X,Y) + g_T(\nabla^T_X \nabla_Y^T X,Y) + g_T(\nabla^T_{\nabla^T_Y X}X,Y) - g_T(\nabla^T_{\nabla_X^T Y}X,Y) \\
&= R^T(Y,X,X,Y) + g_T(\nabla^T_X \nabla_Y^T X,Y) + g_T(\nabla^T_Y X, \nabla^T_Y X) - g_T(\nabla^T_{\nabla_X^T Y}X,Y) \\
&= g_T(\nabla_Y^T X, \nabla_Y^T X) + R^T(Y,X,X,Y) + g_T(\nabla^T_X \nabla^T_Y X,Y) - g_T(\nabla^T_{\nabla^T_X Y} X, Y) \, .
\end{align*}
From the second to the third line we implicitly used that $ \nabla_{[Y,X]}^T X = \nabla_{[Y,X]^\perp}^T X $, since $ X \in \mf{trans}(\mc{F}) $.
\item[c)] If we sum b) over any local orthonormal frame, then the first two terms on the right-hand side yield $ |\nabla^T X|^2 $ and $ \Ric^T(X,X) $, as desired. \\
Fix a point $ x \in M $. As shown in \cite[Section 3]{KTT}, there exists a local orthonormal frame $ E_1, \ldots, E_n $ in a neighborhood of $ x $ such that $ E_1, \ldots, E_p \in \Gamma_\ell(E) $, $ E_{p+1}, \ldots, E_n \in \Gamma_\ell(E^\perp) $ and $ (\nabla^T E_i)_x = 0 $ for $ i = p+1, \ldots, n $. If we sum b) at $ x $ over such a frame, then the last term on the right-hand side vanishes and the third term reduces to
\[ \sum_i g_T(\nabla^T_X \nabla^T_{E_i} X, E_i) = \sum_i X\big(g_T(\nabla_{E_i}^T X, E_i)\big) = X( \Div_T X) = 0 \, . \]
\end{enumerate}
\end{proof}

We can now finally come to our main result:

\begin{theorem} \label{Bochner}
Let $ M $ be a connected closed oriented manifold, endowed with a transversely oriented harmonic Riemannian foliation $ (\mc{F}, g) $. If $ \Ric^T $ is positive semi-definite everywhere, then every basic harmonic vector field is transverse parallel. If additionally $ \Ric^T $ is positive definite at one point, then no non-trivial basic harmonic vector fields exist.
\end{theorem}
\begin{proof}
Let $ X \in \mf{trans}(\mc{F}) $ be a basic harmonic vector field and $ f := \frac{1}{2} \, g_T(X,X) $. By virtue of \autoref{dlem}, \autoref{div} and \autoref{harmeq}:
\[ 0 = \int_M \Delta_T f \cdot \mu = \int_M \Big(|\nabla^T X|^2 + \Ric^T (X,X) \Big) \cdot \mu \geq \int_M |\nabla^T X|^2 \cdot \mu \geq 0 \, . \]
Therefore $ | \nabla^T X|^2 $ vanishes everywhere, meaning $ X $ is transverse parallel. Furthermore, also $ \Ric^T(X,X) $ vanishes everywhere, so if additionally $ \Ric^T $ is positive definite at one point, then $ X $ vanishes at that point. But then $ X $ vanishes everywhere by virtue of \autoref{const}.
\end{proof}

\begin{remark}
Note that \autoref{Bochner} is indeed a generalization of \autoref{posBoch}: If $ \mc{F} $ is the trivial foliation of $ M $ by singletons (i.e.~the corresponding integrable distribution $ E = 0 $), then transverse orientability of $ \mc{F} $ coincides with ordinary orientability of $ M $, the Riemannian foliation $ (\mc{F},g) $ is trivially harmonic and $ \Ric^T = \Ric $. Furthermore, basic harmonic and transverse parallel vector fields are nothing else than ordinary harmonic and parallel vector fields in this case.
\end{remark}

\begin{corollary}
If $ (M,\mc{F}, g) $ are as in \autoref{Bochner} and $ \Ric^T $ is positive semi-definite everywhere, then $ b_1(\mc{F}) \leq q = \codim \mc{F} $. If additionally $ \Ric^T $ is positive definite at one point, then $ b_1(\mc{F}) = 0 $.
\end{corollary}
\begin{proof}
\autoref{Hodge} states that $ b_1(\mc{F}) = \dim \mc{H}_B^1(\mc{F}) $. Fix a point $ x \in M $ and consider the linear map $ \mc{H}_B^1(\mc{F}) \to T_xE^\perp , \, \omega_X \mapsto X_x $. By virtue of \autoref{const} and \autoref{Bochner}, this map is injective, meaning $ b_1(\mc{F}) \leq \dim T_x E^\perp = q $. If additionally $ \Ric^T $ is positive definite at one point, then \autoref{Bochner} even yields $ b_1(\mc{F}) = 0 $.
\end{proof}

We conclude this section by deriving \autoref{tflat}, for which we first need the following
\begin{definition}
\vspace{-.5cm} A transverse vector field $ X \in \mf{trans}(\mc{F}) $ is \emph{transverse Killing} if \mbox{$ \mc{L}_X g_T = 0 $.} We denote the vector space of all transverse Killing fields of $ (\mc{F},g) $ by $ \mf{iso}(\mc{F}) $.
\end{definition}

Again, a transverse vector field $ X \in \mf{trans}(\mc{F}) $ which is Killing in the usual sense that $ \mc{L}_X g = 0 $ is also transverse Killing, but the converse is not true. The same argument as in the proof of \autoref{symm} shows that $ X \in \mf{trans}(\mc{F}) $ is transverse Killing if and only if $ \nabla^T X $ is $ g_T $-skew-symmetric. This also demonstrates that every transverse parallel vector field is transverse Killing. Combining \Cref{tnegBoch,Bochner} yields the following

\begin{corollary} \label{tflat2}
If $ (M,\mc{F}, g) $ are as in \autoref{Bochner} and $ \Ric^T $ vanishes everywhere, then $ \dim \mf{iso}(\mc{F}) = b_1(\mc{F}) $.
\end{corollary}

\section{Applications} \label{Sas}

We conclude this article by applying \autoref{tflat2} to two classes of spaces which naturally satisfy all of the required conditions, namely \emph{degenerate 3-$(\alpha,\delta)$-Sasaki} and certain \emph{Sasaki-$ \eta $-Einstein manifolds}, which arise for example as Boothby-Wang bundles over hyperkähler and Calabi-Yau manifolds, respectively. We only give minimal expositions of these geometries here and refer the interested reader to the comprehensive monograph \cite{BG} as well as the introductory articles \cite{AD} and \cite{ADS} about 3-$ (\alpha,\delta) $-Sasaki manifolds and the recent publication \cite{GRS} which focusses specifically on the degenerate case.

\begin{definition}
Let $ (M^{2n+1}, g, \xi, \eta,\varphi) $ be a Riemannian manifold, endowed with a unit length vector field $ \xi $, its $ g $-dual one-form $ \eta $ and an almost Hermitian structure $ \varphi $ on $ \ker \eta $. Then $ M $ is an \emph{almost contact metric manifold} if
\[ \varphi \,\xi = 0 \, , \qquad \varphi^2 = - \id + \xi \otimes \eta \, , \qquad g \circ (\varphi \times \varphi) = g - \eta \otimes \eta \, . \]
The \emph{Reeb vector field} $ \xi $ spans an integrable distribution $ E $ which defines the so-called \emph{characteristic foliation} $ \mc{F} $. The \emph{fundamental 2-form} is given by $ \Phi(X,Y) := g(X,\varphi Y) $ and $ M $ is a \emph{Sasaki manifold} if $ [\varphi,\varphi] + d\eta \otimes \xi = 0 $ as well as $ d\eta = 2 \Phi $. A Sasaki manifold is called \emph{$ \eta $-Einstein} if its Ricci curvature tensor satisfies $ \Ric = a g + b \eta \otimes \eta $ for some constants $a,b \in \R $.
\end{definition}

We can endow any Sasaki manifold with an orientation and its characteristic foliation with a compatible transverse orientation by using the volume forms $ (d\eta)^n \wedge \eta $ and $ (d\eta)^n$, respectively. The characteristic foliation of any Sasaki manifold harmonic and admits a transverse Kähler structure \cite{BG}. \\

In order to apply \autoref{tflat2} we limit ourselves to the case where the Kähler structure is Ricci-flat, i.e.~Calabi-Yau. In this case the Sasaki manifold is not Einstein in the ordinary sense that $ \Ric = a g $ but instead $ \eta $-Einstein with $ b \neq 0 $ \cite[Theorem 11.1.3]{BG}. Examples of Sasaki-$ \eta $-Einstein manifolds can be constructed via the famous \emph{Boothby-Wang bundle} \cite{BW}:

\begin{theorem}
Let $ N $ be a Calabi-Yau manifold with integral Kähler class. Then a certain $ S^1 $-bundle over $ N $ admits a Sasaki-$ \eta $-Einstein structure.
\end{theorem}

3-$ (\alpha, \delta) $-Sasaki geometry was devised by \textsc{Agricola} and \textsc{Dileo} as a common gene-ralization to accomodate both 3-Sasaki manifolds and other interesting examples like the quaternionic Heisenberg groups \cite{AD}. This new class of manifolds retains favorable properties like hypernormality and canonicity \cite[Theorem 2.2.1 \& Corollary 2.3.3]{AD}.

\begin{definition}
Let $ (M^{4n+3}, g, \xi_i, \eta_i,\varphi_i)_{i=1,2,3} $ be a Riemannian manifold, endowed with three almost contact metric structures. Then $ M $ is an \emph{almost 3-contact metric manifold} if their interrelation is governed by
\begin{gather*}
\varphi_i \, \xi_j=\xi_k \, ,\qquad \eta_i\circ\varphi_j=\eta_k\, , \qquad \varphi_i \circ \varphi_j=\varphi_k+\xi_i\otimes\eta_j
\end{gather*}
for any even permutation $ (i,j,k) $ of $ (1,2,3) $. The three \emph{Reeb vector fields} $ \xi_i $ span an integrable distribution $ E $ which defines the so-called \emph{characteristic foliation $\mc{F} $}. The \emph{fundamental 2-forms} are given by $ \Phi_i(X,Y) := g(X, \varphi_iY) $ and $ M $ is a \emph{3-$(\alpha,\delta)$-Sasaki manifold} if there exist $ \alpha, \delta \in \R $, $ \alpha \neq 0 $, such that
\[ d\eta_i = 2\alpha \Phi_i + 2 (\alpha-\delta) \eta_j \wedge \eta_k \]
for any even permutation $ (i,j,k) $ of $ (1,2,3) $. Finally, $ M $ is called \emph{positive} if $ \alpha \delta > 0 $, \emph{negative} if $ \alpha \delta < 0 $ and \emph{degenerate} if $ \delta = 0 $.
\end{definition}

Again, we can endow any 3-$(\alpha,\delta) $-Sasaki manifold with an orientation and its characteristic foliation with a compatible transverse orientation using $ (d\eta_1)^{2n} \wedge \eta_1 \wedge \eta_2 \wedge \eta_3 $ and $ (d\eta_1)^{2n}$, respectively. Furthermore, the characteristic foliation is harmonic \cite[Corollary 2.3.1]{AD} and we can easily control the sign of the transverse Ricci curvature tensor, cf.~\autoref{curv} \cite[Theorem 2.2.1]{ADS}:

\begin{theorem}
The characteristic foliation of a 3-$(\alpha, \delta) $-Sasaki manifold induces local Riemannian submersions onto a quaternionic Kähler manifold whose Ricci curvature is positive/negative/zero if $ M $ is positive/negative/degenerate.
\end{theorem}

Interesting examples of \emph{degenerate} 3-$(\alpha,\delta)$-Sasaki manifolds can be constructed as \linebreak \emph{3-Boothby-Wang bundles} over hyperkähler manifolds \cite[Theorem 3.1]{GRS}:

\begin{theorem} \label{BW}
Let $ N $ be a hyperkähler manifold with integral Kähler classes. Then a certain $ T^3 $-bundle over $ N $ admits a degenerate 3-$(\alpha,\delta) $-Sasaki structure.
\end{theorem}

As outlined above, \autoref{tflat2} can be applied to both 3-$ (\alpha,\delta) $-Sasaki manifolds and Sasaki-$ \eta $-Einstein manifolds with transverse Calabi-Yau structure:

\begin{corollary} \label{corapp}
Let $ M $ be a connected closed degenerate 3-$(\alpha, \delta) $-Sasaki manifold or Sasaki-$ \eta $-Einstein manifold with transverse Calabi-Yau structure. Then the characteristic foliation $ \mc{F} $ satisfies $\dim \mf{iso}(\mc{F}) = b_1(\mc{F}) $.
\end{corollary}

In this context, one is usually not so much interested in transverse Killing fields per se, but rather in the following notion:

\begin{definition}
An \emph{automorphism of a Sasaki manifold} $ M $ is an isometry which satisfies one of the equivalent conditions $ \phi_\ast \xi = \xi $, $ \phi^* \eta = \eta $ or $ \phi_\ast \circ \varphi = \varphi \circ \phi_\ast $. The collection of all such automorphisms constitutes a Lie group, which we denote by $ \Aut(M) $. The Lie algebra $ \mf{aut}(M) $ of $ \Aut(M) $ is comprised of all complete Killing vector fields $ X $ which satisfy $ \mc{L}_X \xi = 0 $, $ \mc{L}_X \eta = 0 $ and $ \mc{L}_X \varphi = 0 $. \\
An \emph{automorphism of a 3-$(\alpha,\delta)$-Sasaki manifold} $ M $ is an isometry $ \phi:M \to M $ which satisfies one of the equivalent conditions $ \phi_\ast \xi_i = \xi_i $, $ \phi^* \eta_i = \eta_i $ or $ \phi_\ast \circ \varphi_i = \varphi_i \circ \phi_\ast $ for $ i=1,2,3 $. The collection of all such automorphisms constitutes a Lie group, which we denote by $ \Aut(M) $. The Lie algebra $ \mf{aut}(M) $ of $ \Aut(M) $ is comprised of all complete Killing vector fields $ X $ which satisfy $ \mc{L}_X \xi_i = 0 $, $ \mc{L}_X \eta_i = 0 $ and $ \mc{L}_X \varphi_i = 0 $ for $ i = 1,2,3 $.
\end{definition}

From now on, let $ M $ be a connected closed degenerate 3-$(\alpha, \delta) $-Sasaki manifold or Sasaki-$ \eta $-Einstein manifold with transverse Calabi-Yau structure, let $ E $ be the integrable distribution spanned by the Reeb vector field(s) and $ \mc{F} $ the characteristic foliation.

\begin{lemma}
The orthogonal projection of any infinitesimal automorphism to $ E^{\perp} $ is a transverse Killing field.
\end{lemma}
\begin{proof}
Let $ X \in \mf{aut}(M) $ be an infinitesimal automorphism and let $ X^\top $, $ X^\perp $ denote its orthogonal projections to $ E $, $ E^\perp $, respectively. Because $ X $ commutes with the Reeb vector field(s), it follows that $ X \in \mf{fol}(\mc{F}) $. Since $ X^\top $ is trivially foliate, we obtain that $ X^\perp = X - X^\top \in \mf{fol}(\mc{F}) $. This implies that $ (\mc{L}_{X^\perp}g_T)(Y,Z) = 0 $ if $ Y $ or $ Z $ lies in $ \Gamma_\ell(E) $. Furthermore $ \mc{L}_X g = 0 $, since $ X $ is Killing and $ \mc{L}_{X^\top} g_T = 0 $ because $ g_T $ is a transverse metric. Hence, if both $ Y,Z \in \Gamma_\ell(E^\perp) $:
\begin{align*}
(\mc{L}_{X^\perp} g_T) (Y,Z) &= (\mc{L}_X g_T)(Y,Z) \\
&= X\big(g_T(Y,Z)\big) - g_T([X,Y],Z) - g_T(Y,[X,Z]) \\
&= X\big(g(Y,Z)\big) - g([X,Y],Z) - g(Y,[X,Z]) \\
&= (\mc{L}_X g)(Y,Z) = 0 \, .
\end{align*}
\end{proof}

Therefore $ \pi: \mf{aut}(M) \to \mf{iso}(\mc{F}) , \, X \mapsto X^\perp $ is a well-defined linear map. We can determine the kernel of $ \pi $ using an argument from \cite[Lemma 4.2]{GRS}:

\begin{lemma}
The kernel of $ \pi $ is comprised of all the linear combinations of the Reeb vector field(s) with constant coefficients.
\end{lemma}
\begin{proof}
Clearly, all linear combinations of the Reeb vector field(s) with constant coefficients lie in the kernel of $ \pi $. Conversely, let $ X \in \ker \pi = \mf{aut}(M) \cap \Gamma(E) $. Then $ X = f \xi $ or $ X = \sum_{i=1}^3 f_i \xi_i $ with $ f,f_1,f_2,f_3 \in C^\infty(M) $, respectively. For all $ Y \in \Gamma_\ell(TM) $, we have
\[ 0 = \big(\mathcal{L}_X \eta\big)(Y) = f \underbrace{(\mathcal{L}_{\xi} \eta)}_{=0}(Y) + Y(f) \underbrace{\eta(\xi)}_{=1} = Y(f) \, , \]
respectively
\[ 0 = \big(\mathcal{L}_X \eta_j\big)(Y) = \sum_{i=1}^3 \Big(f_i \underbrace{(\mathcal{L}_{\xi_i} \eta_j)}_{=0}(Y) + Y(f_i) \underbrace{\eta_j(\xi_i)}_{=\delta_{ij}}\Big) = Y(f_j) \, . \]
Since $ M $ is connected, it follows that $ f,f_1,f_2,f_3 $ have to be constant.
\end{proof}

Hence, the rank-nullity theorem and \autoref{corapp} yield:
\[ \dim \mf{aut}(M) = \dim \im \pi + \dim \ker \pi \leq \dim \mf{iso}(\mc{F}) + \rk E = b_1(\mc{F}) + \rk E \, . \]
We have thus arrived at our final two theorems:

\begin{theorem}
Let $ M $ be a connected closed degenerate 3-$(\alpha, \delta) $-Sasaki manifold with characteristic foliation $ \mc{F} $. Then the dimension of the automorphism group $ \Aut(M) $ is at most $ b_1(\mc{F}) + 3 $. \\
In particular, if $ M $ arises via \cite[Theorem 3.1]{GRS} as a $ T^3 $-bundle over a compact hyperkähler manifold $ N $ with integral Kähler classes, then $ \dim \Aut(M) \leq b_1(N) + 3 $.
\end{theorem}

\begin{theorem}
Let $ M $ be a connected closed Sasaki-$ \eta $-Einstein manifold with transverse Calabi-Yau structure and characteristic foliation $ \mc{F} $. Then the dimension of the automorphism group $ \Aut(M) $ is at most $ b_1(\mc{F}) + 1 $. \\
In particular, if $ M $ arises as the Boothby-Wang bundle over a compact Calabi-Yau mani- fold $ N $ with integral Kähler class, then $ \dim \Aut(M) \leq b_1(N) + 1 $.
\end{theorem}

\begin{remark}
One might ask if there is even equality $ \dim \Aut(M) = b_1(\mc{F}) + \rk E $ in the above theorems. This is equivalent to the question if $ \pi : \mf{aut}(M) \to \mf{iso}(\mc{F}) $ is surjective or if every transverse Killing field can be extended to an infinitesimal automorphism. In the context of Sasaki manifolds, this problem was further rephrased in \cite[Theorem 8.1.8]{BG}, where they obtain that a transverse Killing field $ X \in \mf{iso}(\mc{F}) $ extends to an infinitesimal automorphism if and only if the basic cohomology class $ [\iota_X d\eta] \in H_B^1(\mc{F}) $ vanishes. The same holds for degenerate 3-$(\alpha,\delta) $-Sasaki manifolds if the classes $ [\iota_X d\eta_i] $ vanish for \mbox{$ i =1,2,3 $.} In the special case $ b_1(\mc{F}) = 0 $ this leads to alternative proofs of the above theorems.
\end{remark}

\printbibliography

\vspace{1cm}

\textsc{Leon Roschig, Philipps-Universität Marburg, Fachbereich Mathematik und Informatik, Hans-Meerwein-Straße 6, 35043 Marburg} \\
\textit{E-mail address}: \texttt{roschig@mathematik.uni-marburg.de}

\end{document}